\documentclass{amsart}
\usepackage{amsmath,amssymb}
\usepackage{graphicx,subfigure}
\usepackage{color}

\newcommand\PA{$\mathrm{P}_{s,n}$}
\newcommand\PB{$\mathrm{P}'_{s,n}$}

\newtheorem{theorem}{Theorem}%[section]
\newtheorem{proposition}[theorem]{Proposition}
\newtheorem{corollary}[theorem]{Corollary}

\newtheorem{lemma}[theorem]{Lemma}

\theoremstyle{definition}
\newtheorem*{defi}{Definition}

\newtheorem*{problemA}{Problem \PA}
\newtheorem*{problemB}{Problem \PB}

\theoremstyle{remark}
\newtheorem{remark}[theorem]{Remark}

\newcommand{\ka}{\kappa}
\newcommand{\la}{\lambda}

\newcommand{\te}{\theta}

\newcommand{\De}{\Delta}

\newcommand{\Om}{\Omega}

\def\RR{\mathbb{R}}

\def\ZZ{\mathbb{Z}}

\newcommand{\cR}{{\mathcal R}}
\newcommand{\cS}{{\mathcal S}}

\newcommand{\pa}{\partial}
\newcommand{\pd}{\partial}
\newcommand\minus\backslash

\newcommand\lan\langle
\newcommand\ran\rangle

\def\th{^{\text{th}}}

\renewcommand\leq\leqslant
\renewcommand\geq\geqslant
\newlength{\intwidth}

\newcommand\Dir{^{\mathrm{Dir}}}
\newcommand\Neu{^{\mathrm{Neu}}}

\begin{document}

\title{Spectral determination of semi-regular polygons}

\author{Alberto Enciso}
\address{Instituto de Ciencias Matem\'aticas, Consejo Superior de
  Investigaciones Cient\'\i ficas, 28049 Madrid, Spain}
\email{aenciso@icmat.es}

\author{Javier G\'omez-Serrano}
\address{Department of Mathematics, Princeton University, Princeton,
  NJ 08544, USA}
\email{jg27@math.princeton.edu}

%%    General info
%\subjclass[2010]{35B38, 58J05, 58K45}
%\date{\today}
%
%\keywords{ }
%
\begin{abstract}
Let us say that an $n$-sided polygon is semi-regular if it is circumscriptible
 and its angles are all equal but possibly one, which is then larger
 than the rest. Regular polygons, in particular, are semi-regular.
We prove that semi-regular polygons  are spectrally determined in the class of
convex piecewise smooth domains. Specifically, we show that if $\Om$
is a convex piecewise smooth planar domain, possibly with straight corners,
whose Dirichlet or Neumann spectrum coincides with that of an $n$-sided
semi-regular polygon~$P_n$, then $\Om$ is congruent to~$P_n$.
\end{abstract}
\maketitle

\section{Introduction}

The inverse spectral problem for a bounded planar domain~$\Om_0$ is to
ascertain whether any domain $\Om\subset\RR^2$ with the same
spectrum (with, say, Dirichlet boundary conditions) is actually congruent to~$\Om_0$. Since Marc Kac reformulated this
problem in 1965 as ``hearing the shape of a drum''~\cite{Kac}, the
spectral determination of planar domains (and of more general
Riemannian manifolds, with or without boundary) has become a major research topic for analysts
and geometers. We recall that the Dirichlet spectrum of a planar
domain~$\Om$ is the set (with multiplicities) of numbers~$\la_k$ for which the
boundary value problem
\[
\De u_k+\la_k u_k=0\quad \text{in }\Om\,,\qquad u_k|_{\pd\Om}=0\,,
\]
admits a nontrivial solution. The Neumann eigenvalues $\mu_k$ of a domain are
similarly defined in terms of the boundary value problem
\[
\De v_k+\mu_k v_k=0\quad \text{in }\Om\,,\qquad \pd_\nu v_k=0\,,
\]

Most of the existing results of the inverse spectral
problem are negative, meaning that they assert that a certain domain
or manifold cannot be recovered from its Dirichlet spectrum. In
particular, it is known since the 1980s~\cite{Gordon} that there are polygons that are isospectral but not isometric. The question
of whether there are any smooth or convex planar domains with this
property remains open. In the case
of Riemannian manifolds it is known that even the local geometry of isospectral
manifolds can be different~\cite{Szabo}. 

The list of positive results for the inverse spectral problem remains
rather short. In the case of negatively curved manifolds, there are
spectral rigidity~\cite{Croke,Guillemin} and compactness
results~\cite{Sarnak}, some of which have been recently extended to
Anosov surfaces~\cite{Uhlmann}. These compactness results are also
valid for planar domains~\cite{Sarnak2,Sarnak3}. A major
breakthrough has been the proof that, in the class of analytic bounded planar
domains with a  reflection symmetry, any domain is spectrally
determined~\cite{Zelditch}. This result, which hinges on the computation of all the
wave invariants at a closed billiard trajectory formed by a segment
that hits the boundary orthogonally at both endpoints, has been
extended to higher dimensional domains in~\cite{Zelditch2}. 

It is very easy to show~\cite{Kac} that the any planar domain with the
same Dirichlet spectrum as a disk must be congruent to it. This
follows from the facts that disks are the only minimizers of the
isoperimetric inequality
\[
|\pd\Om|^2\geq 4\pi |\Om|
\]
and that the area and length are obtained from the short-time asymptotics of the
heat trace
\begin{equation}\label{hOm}
h\Dir_\Om(t):=\sum_{k=1}^\infty e^{-\la_k t}\,.
\end{equation}
The result remains true for balls of any dimension and for Neumann
boundary conditions. Quite
remarkably, the somewhat reminiscent problem of the spectral determination of round spheres is only known
up to dimension~6~\cite{Tanno}. Other than the disks, the only planar
domain that is known to be spectrally determined without any symmetry
and analyticity assumptions is a family of ovals with four vertices
whose isoperimetric quotients are very close to that of a disk~\cite{Watanabe}.

Our objective in this paper is show that regular polygons (that is,
those whose angles are all equal and whose sides are all of the same length) are also
determined by their Dirichlet or Neumann spectrum in the class of
convex piecewise smooth domains. More generally, the
result holds for the polygons that we will call {\em semi-regular}\/, meaning circumscriptible polygons whose angles are all
equal but possibly one, which must then be larger than the rest:

\begin{theorem}\label{T.main}
  Let $P_n$ be a semi-regular $n$-sided polygon. If $\Om$ is a convex piecewise smooth bounded planar domain, possibly
  with straight corners, and its Dirichlet or Neumann spectrum
  coincides with that of~$P_n$, then $\Om$ is congruent to~$P_n$.
\end{theorem}

Of course, a planar domain is said to be {\em piecewise smooth}\/ if
its boundary is a $C^\infty$~curve but possibly at a finite number of
points, which we call {\em corner points}\/. We assume that the
corners are {\em straight}\/, i.e., that the boundary is flat in a
small neighborhood of each corner point. 
% domain
% satisfies an interior point condition at the corner points (further
% details are given in Section~\ref{S.spectral}), which prevents the
% existence of cusp singularities in the boundary, but we do not assume
% that the corners are straight (i.e., that the boundary is flat at the
% corner points). We do not need any convexity assumptions. 
As a
side remark, notice that, when considering the inverse spectral
problem for regular polygons, it is essential to allow for piecewise
smooth domains because a domain with corners cannot be isospectral to
a domain with smooth boundary~\cite{Rowlett}.

% where,
% as a technical assumption. We are
% also making the assumption. Notice
% that in the theorem we have made the technical assumption that at the
% points where the boundary is non-differentiable the domain has {\em straight corners}\/,
% meaning that all the lateral derivatives of the curvature at the point are zero (so both lateral boundary curves are
% flat a this point) and that the domain satisfies an
% interior cone condition at this point. This prevents the domain, for instance,
% from having cusp singularities.

Let us discuss the main ideas of the proof of the theorem in the case
of Dirichlet boundary conditions, the Neumann case being completely
analogous. The key
step of the proof is to prove the result for convex polygons, that is, to show that if
$P$~is a polygon isospectral to an $n$-sided semi-regular polygon~$P_n$, then
$P$ and $P_n$ are congruent. The reason is that the asymptotic
behavior at~0 of the heat trace~\eqref{hOm} of a piecewise
domain~$\Om$ with $m$ straight corners is, as shown in
Proposition~\ref{P.heat} using results available in the literature,
\begin{equation*}%multline*}
h_\Om\Dir(t)=\frac{|\Om|}{4\pi t} -\frac{|\pd\Om|}{8\sqrt{\pi
    t}}
+\frac1{12\pi}\bigg(\int_{\pd\Om}\ka\,ds+\sum_{k=1}^m\frac{\pi^2-\te_k^2}{2\te_k}\bigg)%\\
+\frac{\sqrt t}{256\sqrt\pi} \int_{\pd\Om}\ka^2\, ds + O(t)\,,
\end{equation*}
where $\te_k\in(0,2\pi)$ is the interior angle at the $j\th$ corner
point and $\ka$ is the curvature of the boundary at each
differentiable point. From the expression of the term of order
$\sqrt t$ it is not hard to see that the boundary of~$\Om$ must be
flat, which means that $\Om$ must indeed be a polygon. Notice that the
only reason for which we have made the technical assumption that the
corners of~$\Om$ are straight is simply that the fourth coefficient of
the asymptotic expansion of the heat trace has apparently not yet been
computed for domains with non-straight corners. In the case of
straight corners and Dirichlet boundary boundary conditions, the
result was first obtained (but not published) by Dan Ray, as
referenced by Kac~\cite{Kac} and later by Cheeger~\cite{Cheeger}, while for
Neumann boundary conditions the result has appeared only very recently~\cite{Mazzeo}.

To establish the result for polygons we derive a new characterization
of semi-regular polygons as extremizers (in the space of convex polygons
whose number of sides is not fixed) of a kind of constrained
isoperimetric inequality which, to our best knowledge, had not
appeared in the literature before. The constraint involves a third
geometric quantity that is a spectral invariant and whose expression,
which is highly nontrivial from a geometric standpoint, can only be
guessed using spectral-theoretic methods. Let us recall that an
$n$-sided regular polygon is the $n$-gon with fixed area and minimal
perimeter, but regular polygons are {\em not}\/ extremizers of the
isoperimetric inequality when one allows for polygons with an
arbitrary number of sides, and one does not know how to find the
number of sides of a polygon using spectral invariants. Hence the key
advantage of our isoperimetric-type characterization of regular (and,
more generally, semi-regular)
polygons is that it enables us to characterize an $n$-sided semi-regular
polygon among polygons with an arbitrary number of sides. It should be
noticed that the hypothesis that the domains are convex is only used
in this part of the proof.

Incidentally, let us recall that, for a fixed number of sides,
the problem becomes effectively finite-dimensional, so finer results
can sometimes be proved. For example, it is easy to prove that any
rectangle is characterized among rectangles by its spectrum. The
same result is true in the case of triangles, but the proof is much
harder (see~\cite{Durso} for the original proof based on wave
invariants and~\cite{Grieser} for a simpler proof that only needs the
heat kernel). Likewise, trapezoids have been recently
shown~\cite{Rowlett2} to be determined by their Neumann spectrum within the class
of trapezoids. The proof makes essential use of wave invariants to
obtain an additional geometric quantity, which does not appear in the
heat trace and is crucially used to determine the parameters that
characterize an arbitrary trapezoid.

The paper consists of two sections corresponding to the two parts of
the proof that we have described above. The proof that any polygon
isospectral with a semi-regular polygon is indeed congruent to it,
which is the heart of the paper, is presented in
Section~\ref{S.polygons}. The proof of Theorem~\ref{T.main} then
follows after showing that any domain isospectral with a polygon must
be a polygon, which is established in Section~\ref{S.strategy}.

\section{Semi-Regular polygons as maximizers}
\label{S.polygons}

Our objective in this section is to show that any $m$-sided
polygon~$P$ that is isometric to an $n$-sided semi-regular polygon is in
fact congruent to it. 

For this we will rely on the well known fact (see~\cite{Berg} for a
transparent derivation) that the asymptotic behavior of the
heat kernel of a polygon~$P$ can be completely characterized modulo exponentially small
terms as
\[
h_P\Dir(t)=\frac{|P|}{4\pi t} -\frac{|\pd P|}{8\sqrt{\pi
    t}}
+\frac{S(P)}{24\pi}+ O(e^{-c/t})\,,
\]
where $|P|$ and $|\pd P|$ respectively denote the area and
perimeter of~$P$ and
\[
S(P):=\sum_{k=1}^m\frac{\pi^2-\te_k^2}{\te_k}\,.
\]
In the Neumann case, where the heat trace is defined as
\[
h\Neu(t):=\sum_{k=1}^\infty e^{-\mu_k t}\,,
\]
the formulas are identical modulo a sign reversal~\cite{Mazzeo}:
\[
h_P\Neu(t)=\frac{|P|}{4\pi t} +\frac{|\pd P|}{8\sqrt{\pi
    t}}
+\frac{S(P)}{24\pi}+ O(e^{-c/t})\,,
\]

Hence our goal in this section is to characterize semi-regular polygons as the maximizers
of a functional that one can construct using only the three quantities
that appear in the heat trace asymptotics of a polygon: the area, the
length and the function $S(P)$. Furthermore, we can use dilations to restrict our attention to
polygons with unit perimeter. This leaves us just the area $|P|$
and the function $S(P)$, and we will define the maximization problem
as follows:

\begin{defi}
We will say that a convex polygon $P$ with unit perimeter is a {\em
  maximizer}\/ if for every convex polygon $Q$ with unit perimeter and $S(Q)=S(P)$ one has
$|Q|\leq |P|$.
\end{defi}

\begin{figure}
\centering
\includegraphics[scale=0.4]{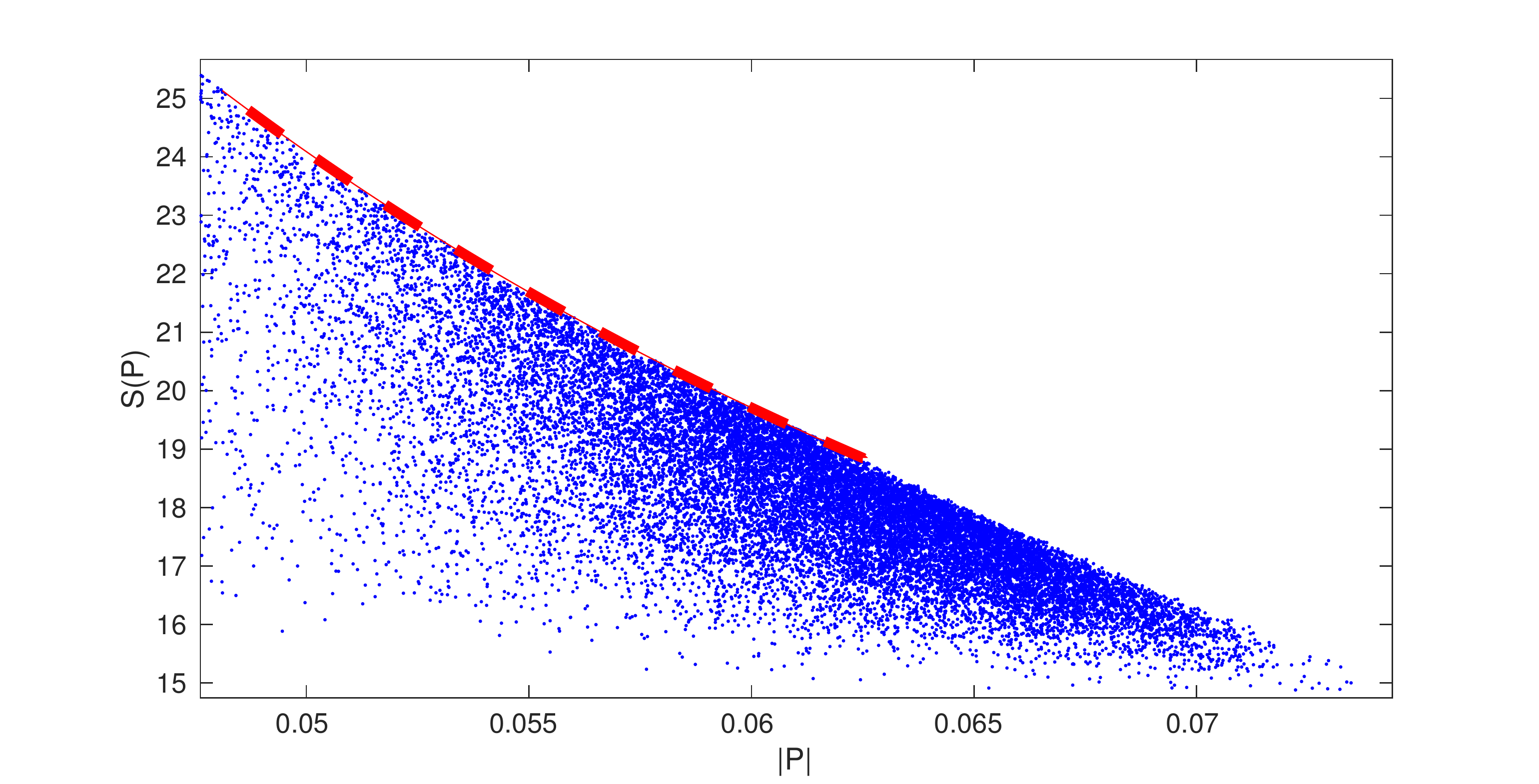}
\caption{Plot of the quantities $(|P|,S(P))$ for 20,000 randomly generated
  convex polygons of unit perimeter. The dashed line corresponds to the semiregular 4-gons, the endpoints being the equilateral triangle and the square.}
\label{F.Semiregular}
\end{figure}

Notice that a maximizer is an extremizer of the isoperimetric quotient
among polygons subject to the constraint that the value of the
function~$S$ is fixed. The key ingredient in the proof of
Theorem~\ref{T.main} is the following result, which shows that
semi-regular polygons are the only maximizers. For the benefit of the reader, before presenting the statement and
proof of this result it is worth to illustrate this
result with some numerical results to get some intuition. In Figure~\ref{F.Semiregular} we
have plotted the values $(S(P),|P|)$ for 20,000 randomly generated
convex polygons of unit perimeter that we have obtained using Monte
Carlo simulations. On top of this
we have plotted the curve $(S(P),|P|)$ corresponding to the semi-regular
polygons.

\begin{theorem}\label{T.maximizers}
A convex $n$-sided polygon of unit perimeter is a maximizer if and only if it is
circumscriptible and all its angles are equal but possibly one, which
must then be larger than the rest. Moreover, if
$P$ and $Q$ are two non-congruent polygons as above, possibly with a different
number of sides, then $S(P)\neq S(Q)$.
% of size $\theta_1 \in \left[\frac{n-2}{n}\pi,\pi\right)$
\end{theorem}

\begin{proof}
We will show that if a polygon $P$ is a maximizer, then it needs to be
of that form, and that the maximizer is completely determined by the
value of the function~$S$.

The first observation is that while looking for maximizers we can
restrict ourselves to the class of circumscriptible polygons. This is
by a classical result of L'Huilier~\cite{LHuilier}, extended to the
case of polygons on the sphere by Steiner~\cite{Steiner}, which
asserts that among all convex $n$-sided polygons of unit perimeter
with given angles, only the one circumscribed to a circle has the
largest area. 

We can therefore write all the quantities in terms of
the interior angles of the polygon
$\theta_1, \theta_2, \ldots, \theta_n$, now assumed to be
circumscribed to a circle (see Figure \ref{F.Circumscriptible}), which yields the following formulas after
imposing that $|\pd P|=1$:

\begin{align*}
S(P) = \sum_{k=1}^n\frac{\pi^2-\te_k^2}{\te_k}, \quad
|P| = \frac{1}{4}\left(\sum_{k=1}^n\cot\frac{\te_k}{2}\right)^{-1}
\end{align*}

\begin{figure}
\centering
\includegraphics[scale=0.4]{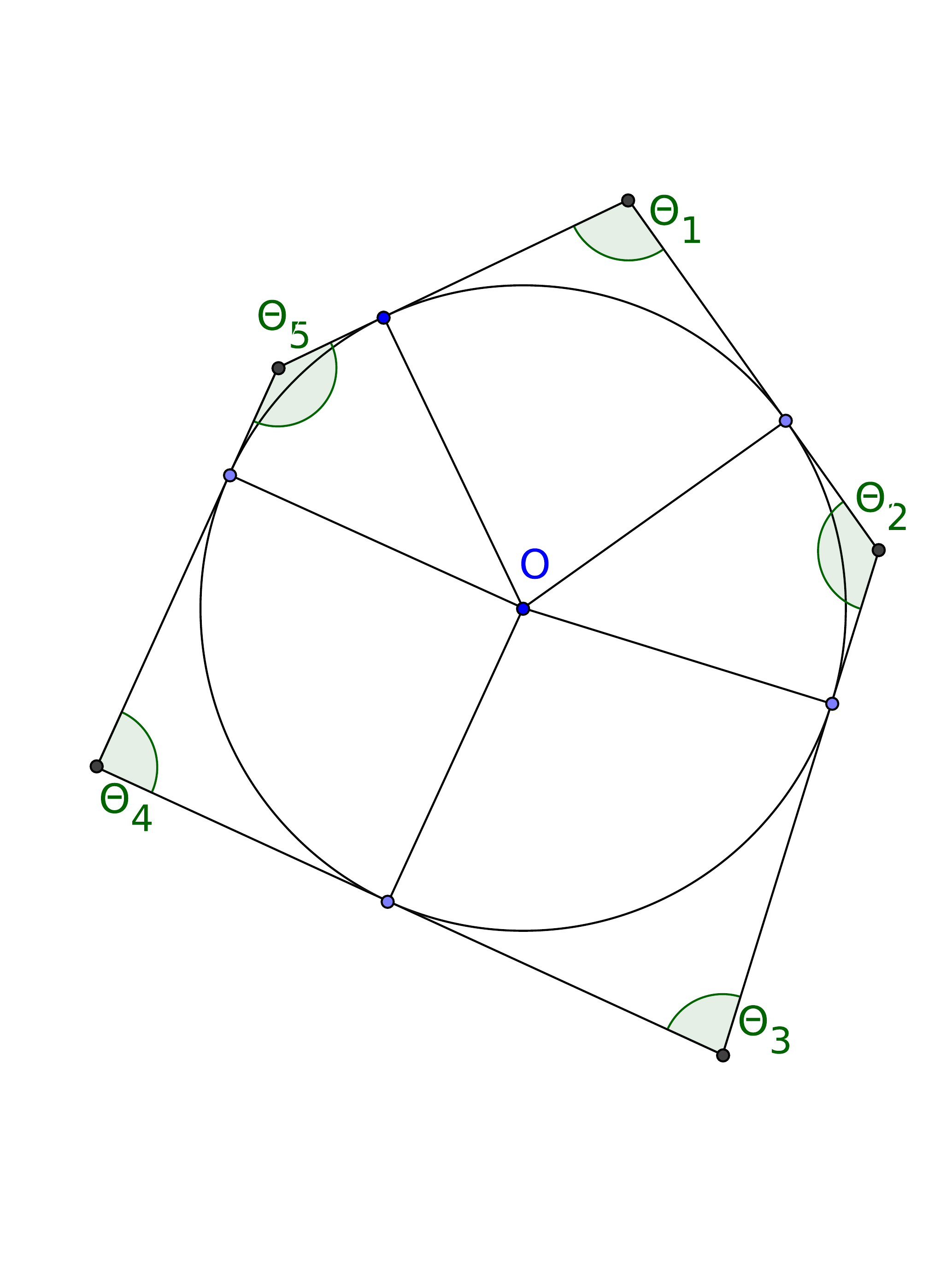}
\caption{A circumscriptible pentagon}
\label{F.Circumscriptible}
\end{figure}

We want to fix the value of $S(P)$, 
which for convenience we will denote by $4\pi\frac{s-1}{s-2}$ with $s
> 2$. The reason for choosing this normalization is that the value of the function~$S$ on the regular $n$-sided
polygon is precisely $4\pi\frac{n-1}{n-2}$. We also know
 that the sum of the angles
of an $n$-sided polygon must satisfy
\begin{equation}\label{cond1}
\sum_{i=1}^{n}\theta_i  =  (n-2)\pi.
\end{equation}
We will consider the following problem, where the unknown is the set of numbers $(\te_1,\dots, \te_n)\in
(0,\pi)^n$ subject to the constraints that correspond to setting
\[
S(P)=4\pi\frac{s-1}{s-2}
\]
and imposing the geometric condition~\eqref{cond1}:

\begin{problemA}%[$A_{m,n}$]
Maximize the quantity
\begin{align*}
\left(\sum_{k=1}^n\cot\frac{\te_k}{2}\right)^{-1}
\end{align*}
with $(\te_1,\dots, \te_n)\in (0,\pi)^n$ subject to the constraints
that
\begin{align*}
\sum_{k=1}^n\frac{1}{\te_k}  =  \frac{1}{\pi}\left(n+\frac{2s}{s-2}\right) \,,\qquad \sum_{i=1}^{n}\theta_i  =  (n-2)\pi\,.
\end{align*}
\end{problemA}

We will show that there is a unique solution to this problem, which
defines a circumscriptible polygon of $\lceil s \rceil$ sides whose
angles are all equal but possibly one, which is then larger than the
rest. To do so, we proceed by induction in $n$. 

We start doing the base case $n = \lceil s \rceil$, as $n$ cannot be
smaller than $\lceil s \rceil$ because of the isoperimetric
inequality for polygons. Using the inequality~\cite[Remark 1.4,
Corollary 1.6]{Cirtoaje}, all we have to do in this case is to check
that the function
\begin{equation*}%\label{fx}
g(x):= f'(x^{-\frac12})
\end{equation*}
is concave in the interval
$\left(\pi^{-2},\infty\right)$, where
$$
f(x) := \cot\frac{x}{2}\,.
$$
Since
\begin{align*}
g(x)= \frac{1}{\cos x^{-\frac12}-1}\,,
\end{align*}
its second derivative is
\begin{align*}
g''(x) = 
- \frac{2+\cos x^{-\frac12}
  - 3 x^{\frac12}\sin x^{-\frac12}}{16 x^3\sin^4(\frac12x^{-\frac12})}.
\end{align*}
Thus to prove that the function~$g$ is concave it suffices to
show that
\begin{align*}
z\,(2 + \cos z) - 3\sin z > 0
\end{align*}
for all $z\in(0,\pi)$.

Using a sixth order Taylor expansion one readily sees that, for all
$z\in (0,\pi)$,
\begin{align*}
z\, (2+\cos z) &= z\left(3-\frac{z^{2}}{2} + \frac{z^4}{24} -
                  \frac{z^{6}}{720} +\frac1{5040}\int_0^z (z-\zeta)^6
  \sin \zeta\, d\zeta \right) \\
&> z\left(3-\frac{z^{2}}{2} + \frac{z^4}{24} -
                  \frac{z^{6}}{720} \right) \,,
\end{align*}
and by a similar argument
\[
3\sin z  < 3z - \frac{z^{3}}{2} + \frac{z^{5}}{40} 
\]
on $(0,\pi)$. This ensures that, for all $z$ in $(0,\pi)$,
\begin{align*}
z\, (2+\cos z) - 3\sin z & > \frac{z^{5}}{60} - \frac{z^{7}}{720},
\end{align*}
which is strictly positive on this interval. The case $n = \lceil
s \rceil$ then follows. 
%\color{black}.

We now move on to establish the inductive step. We can therefore take
some~$n$, assume that the
only maximizer of Problem P$_{s,n'}$ for any $n' < n$ is the
aforementioned polygon, and show that the same remains true for Problem~P$_{s,n}$. To this end we shall next show that we can reduce Problem~P$_{s,n}$ to a problem
of the form P$_{s,n'}$ with $n' < n$. 

Using Lagrange multipliers, we need to minimize the function
\begin{multline*}
F(\theta_1,\ldots,\theta_n,\lambda_1,\lambda_2) :=
\sum_{k=1}^n\cot\frac{\te_k}{2} +
\lambda_1\left(\sum_{k=1}^n\frac{1}{\te_k}  -\frac{n}{\pi}-
  \frac{2s}{\pi(s-2)}\right) \\
+ \lambda_2\left(\sum_{i=1}^{n}\theta_i  -  (n-2)\pi\right)
\end{multline*}
whose minima for fixed~$(\la_1,\la_2)$ and $(\te_1,\dots,\te_n)\in
(0,\pi)^n$ must be either attained at interior critical points or lie
on the boundary. 

We first deal with the boundary case. There are two possibilities:
either $\theta_i = 0$ for some $i$ or $\theta_i = \pi$ for some
$i$. In the former case the function $F$ will go to $+\infty$ and thus
it will not be
a minimum. In the latter, having a minimum at a point of the form 
\begin{equation}\label{valepi}
(\theta_1,\ldots,\theta_{i-1},\pi,\te_{i+1}, \ldots, \theta_n)
\end{equation}
of Problem~P$_{s,n}$ is equivalent to have a minimum
$(\theta_1,\ldots,\theta_{i-1},\theta_{i+1}, \ldots, \theta_n)$ of
Problem~P$_{s,n-1}$, which is ruled out by the induction hypothesis. 

Therefore, it boils down to controlling the interior critical points
of the function~$F$. Taking the partial derivative with respect to
$\theta_i$, we infer that if there is a local minimum
at~$(\te_1,\dots,\te_n)$ then one necessarily has, for $1\leq
i\leq n$,
\begin{align*}
0 = \frac{\partial F}{\partial \theta_i}(\te_1,\dots,\te_n,\la_1,\la_2) = \Phi(\te_i,\la_1,\la_2)\,,
\end{align*}
where
\[
\Phi(z,\la_1,\la_2):= -\frac{1}{2\sin^{2}\frac{z}{2}} - \frac{\lambda_1}{z^{2}} + \lambda_2\,.
\]

The key observation here is that one gets the same equation for all
the angles, which only depends on one of the variables at a time and
on the Lagrange parameters. Furthermore, it was proved in~\cite[Lemma
1(b)]{Grieser} that the function $\Phi(\cdot,\la_1,\la_2)$ has at most
two zeros on the interval~$(0,\pi)$ for any value
of~$(\la_1,\la_2)$. This implies that if $(\te_1,\dots,\te_n)$ is a critical point then the angles~$\te_i$ can take at most two distinct values. 

Let us assume that the two distinct values are $\bar\te_1$ and
$\bar\te_2$ and that there are $k$ and $n-k$ copies of each angle
respectively. Problem P$_{s,n}$ is reduced to the following:

\begin{problemB}%[$B_{m,n}$]
Maximize the function
\begin{align*}
\left(k\cot\frac{\bar\te_1}{2}+(n-k)\cot\frac{\bar\te_2}{2}\right)^{-1}
\end{align*}
with $(\bar\te_1,\bar\te_2)\in (0,\pi)^2$ subject to the constraints
that
\begin{align*}
\frac{k}{\bar\te_1} + \frac{n-k}{\bar\te_2}  =  \frac{1}{\pi}\left(n+\frac{2s}{s-2}\right) \,,\quad k \bar\te_1 + (n-k) \bar\te_2  =  (n-2)\pi
\end{align*}
\end{problemB}

Solving the constraints for $(\bar\te_1, \bar\te_2)$ in terms of
$(k,n)$ and plugging the resulting expressions into the function we want to maximize  yields
very complicated formulas, so instead we will solve for $(k,n)$ as
functions of $(\bar\te_1,\bar\te_2)$. This is much simpler
because the constraints depend linearly on~$k$ and~$n$, so we get
\begin{align*}
n(s,\bar\te_1,\bar\te_2) &= \frac{2(s-2)\pi^2 - 2\bar\te_1\bar\te_2s}{(s-2)(\pi-\bar\te_1)(\pi-\bar\te_2)}\,, \\
k(s,\bar\te_1,\bar\te_2) &= \frac{2\bar\te_1(\bar\te_2 s - (s-2)\pi)}{(s-2)(\bar\te_2-\bar\te_1)(\pi-\bar\te_1)}\,.
\end{align*}
\color{black}
This allows us to get rid of the constraints and minimize the objective function
\begin{align*}
G_s(\bar\te_1,\bar\te_2) := k(s,\bar\te_1,\bar\te_2)\cot\frac{\bar\te_1}{2}+(n(s,\bar\te_1,\bar\te_2)-k(s,\bar\te_1,\bar\te_2))\cot\frac{\bar\te_2}{2}
\end{align*}
with $(\bar\te_1,\bar\te_2)\in (0, \pi)^2$. 

We first argue that $k$ and $n-k $ must be both larger than or equal
to~1. Indeed, if one assumes that $k = 0$, that implies that
$\bar\te_2 = \frac{n-2}{n}\pi$, with solution $n = s$ whenever $s$ is
an integer and no solution otherwise. This case is discarded because we are assuming that $n \geq \lceil
s \rceil +1$. A similar argument shows that $n-k\geq1$.

Because of the symmetries of the problem (that is, the fact that
exchanging $\bar\te_1$ for $\bar\te_2$ is equivalent to exchanging $k$
for $n-k$), it is enough to consider the case $\bar\te_1 \geq \bar\te_2$. The next
step is to reduce the region where critical points can lie. First, we
notice that $k(s,\bar\te_1,\bar\te_2) \geq 0$ if and only if
\begin{equation}\label{rest1}
\bar\te_2 \leq \frac{s-2}{s}\pi\,,
\end{equation}
and straightforward calculations
show that
$$
k(s,\bar\te_1,\bar\te_2) \leq n(s,\bar\te_1,\bar\te_2)
$$
if and only if 
\begin{equation}\label{rest2}
\bar\te_1 \geq \frac{s-2}{s}\pi\,.
\end{equation}
If we further impose the condition 
$$
n(s,\bar\te_1,\bar\te_2) \geq s + 1\,,
$$
a short computation shows that
\begin{equation}\label{rest3}
\bar\te_1\geq h(s,\bar\te_2)\,,
\end{equation}
where
\begin{align*}
%h(s,\bar\te_2)&:=  \frac{\pi (s-2)(\bar\te_2(1+s)+\pi(1-s))}{2\pi + (s-1)(\bar\te_2(2+s)-s\pi)} \,.
h(s,\bar\te_2)&:=  \frac{\pi (s-2)(\pi(s-1)-\bar\te_2(1+s))}{2\bar\te_2 + (s+1)(\pi(s-2)-s\bar\te_2)} \,.
\end{align*}
\color{black}
Notice that $h(s,\bar\te_2)$ is
an increasing function of $\bar\te_2$ in the
interval~$\left[0,\frac{s-2}{s}\pi\right]$, so this readily gives the uniform bound
\begin{align*}
\bar\te_1 \geq h(s,0) = \frac{s-1}{s+1}\pi
\end{align*}

Due to the constraints on~$(\bar\te_1,\bar\te_2)$ imposed
by~\eqref{rest1}--\eqref{rest3}, we will look for maximizers of
Problem~P$_{s,n}'$ in the bounded planar domain~$\cS_s$ whose
boundary is given by the segments
\begin{align*}
\Gamma^{\cS_s}_{1} & := \bigg\{\bar\te_1 = \frac{(s^2-s-2)\pi}{ s^2+s-2} - \frac{4 \pi^2 ( s -2)}{(s^2+s-2) (\bar\te_2(s^2+s-2) + \pi(-s^2+s+2))}, \\
&\qquad \qquad \qquad \qquad\qquad \qquad \qquad \qquad \qquad\qquad \qquad \qquad \qquad\quad\;\bar\te_2 \in \left[0,\frac{s-2}{s}\pi\right]\bigg\} \\[1mm]
\Gamma^{\cS_s}_{2} &  := \left\{\bar\te_1 \in
                   \left[\frac{s-1}{s+1}\pi,\; \pi\right], \bar\te_2 = 0\right\}, \quad
\Gamma^{\cS_s}_{3} := \left\{\bar\te_1 = \pi, \;\left[0,\frac{s-2}{s}\pi\right]\right\}.
\end{align*}
\color{black}
For computational purposes, it is often convenient to consider the larger
region $\cR_s\supset\cS_s$ (see Figure~\ref{F.Gamma}) that is bounded by the segments
\begin{align*}
\Gamma^{\cR_s}_{1} & := \left\{\bar\te_1 = \frac{s-1}{s+1}\pi, \;\bar\te_2 \in \left[0,\frac{s-2}{s}\pi\right]\right\}, \quad
\Gamma^{\cR_s}_{2}  := \left\{\bar\te_1 \in \left[\frac{s-1}{s+1}\pi,\pi\right],\; \bar\te_2 = 0\right\} \\[1mm]
\Gamma^{\cR_s}_{3} & := \left\{\bar\te_1 = \pi, \;\bar\te_{2} \in \left[0,\frac{s-2}{s}\pi\right]\right\}, \quad
\Gamma^{\cR_s}_{4}  := \left\{\bar\te_1 \in \left[\frac{s-1}{s+1}\pi,\;
                   \pi\right], \bar\te_2 = \frac{s-2}{s}\pi\right\}\,.
\end{align*}

\begin{figure}
\centering
\includegraphics[scale=0.375]{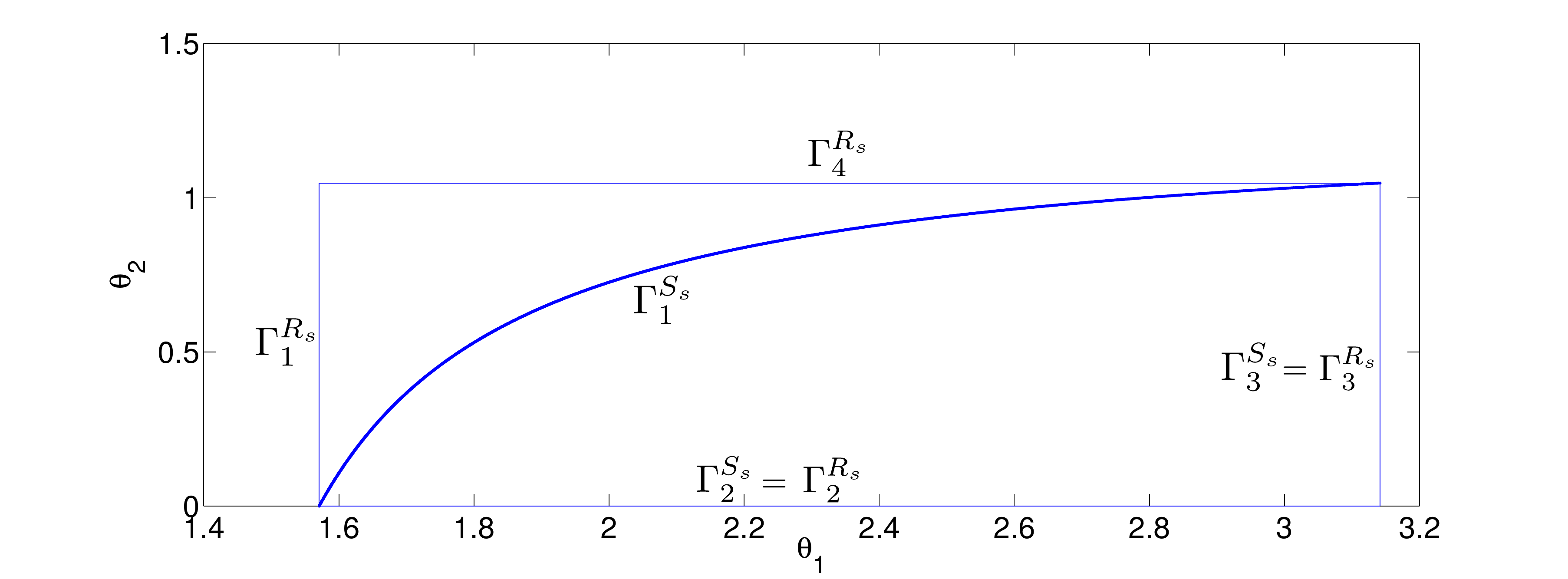}
\caption{The regions $\cR_{3}$ and $\cS_3$}
\label{F.Gamma}
\end{figure}

We start by showing that the function~$G_s$ does not have any critical points in the
interior of~$\cR_s$. The key fact is contained in the following
proposition, whose rather involved proof contains two auxiliary lemmas:

\begin{proposition}\label{propderivadaGpositiva}
The derivative of the function~$G_s$ with respect to the first
variable satisfies
\begin{align*}
\pa_{\bar\te_1} G_s(\bar\te_1,\bar\te_2) > 0&\qquad  \text{in the
                                              interior of }\cR_s\,,\\
\pa_{\bar\te_1} G_s(\bar\te_1,\bar\te_2) > 0&\qquad  \text{on }\Gamma^{\cR_s}_{2}\,, \\
\pa_{\bar\te_1} G_s(\bar\te_1,\bar\te_2) = 0&\qquad  \text{on }\Gamma^{\cR_s}_{4}\,.
\end{align*}
\end{proposition}

\begin{proof}
Taking the derivative yields
\begin{align*}
\pa_{\bar\te_1} G_s(\bar\te_1,\bar\te_2) & =
\frac{ (s-2) \pi -\bar\te_2 s}{(\bar\te_1 - \bar\te_2)^2 (s-2)
                                            (\pi -
                                           \bar\te_2)} \, \widetilde{G}_s(\bar\te_1,\bar\te_2)\,,
\end{align*}
where 
\begin{multline*}
 \widetilde{G}_s(\bar\te_1,\bar\te_2)  := \frac{1}{(\pi-\bar\te_1)^2}
\bigg[2 \bar\te_2 (\bar\te_1 - \pi)^2 \cot\frac{\bar\te_2}{2}
\\
+ (\pi - \bar\te_2) \csc^2\frac{\bar\te_1}{2}\big(\bar\te_1(\bar\te_1
- \bar\te_2) (\bar\te_1 - \pi) + (\bar\te_1^2 - \bar\te_2 \pi) \sin \bar\te_1\big)\bigg]
\end{multline*}
This automatically shows that the derivative vanishes
on~$\Gamma^{\cR_s}_{4}$. Since the first factor is always positive in
the interior of $\cR_s$ and on $\Gamma^{\cR_s}_{2}$, we only need to show that
$\widetilde{G}_s(\bar\te_1,\bar\te_2) >0$ in~$\cR_s$.

We will next show that $\widetilde{G}_s$ is increasing in
$\bar\te_1$, which will be enough to finish the proof once we show that $\widetilde{G}_s(\frac{s-1}{s+1}\pi,\bar\te_2) > 0$. We start by taking a derivative
with respect to $\bar\te_1$, which gives us
\begin{align*}
\pa_{\bar\te_1} \widetilde{G}_s(\bar\te_1,\bar\te_2) & =
\left[(\bar\te_1 - \bar\te_2)(\pi-\bar\te_2)\csc ^{2}\left(\frac{\bar\te_1}{2}\right)\right]
                                                       \big(
                                                       T_1(\bar\te_1)- T_2(\bar\te_1) \big)\,,
\end{align*}
where
\begin{align*}
T_1(x) &:= \frac{x}{\pi-x} \cot\frac{x}{2} \,,\\
T_2(x) &:= -\frac{2\pi}{(\pi-x)^{3}}(x - \pi + \sin x)\,.
\end{align*}
The factor in square brackets is always strictly positive. We shall now show that
the second factor is also strictly positive. We remark that this
factor, that is $T_1(\bar\te_1)- T_2(\bar\te_1)$, is independent
of~$s$ and~$\bar\te_2$, so we are left with two functions of just one variable.

%[Hay que probar que T1 > T2]

\begin{lemma}\label{lemmaT1T2}
$T_1(x) \geq T_2(x)$ for all $x \in [0,\pi]$, with equality only at $x = 0$.
\end{lemma}

\begin{proof}
An easy computation shows that $T_1(0) = T_2(0), T_1'(0) = T_2'(0)$. We will show that $T_1''(x) > 0$ and $T_2''(x) \leq 0$ for all $x \in [0,\pi]$.

Expanding the cotangent in partial fractions~\cite[4.3.91]{Abramowitz}
as
\[
\cot z=\frac1z +2z \sum_{k\geq1} \frac1{z^2-(k\pi)^2}
\]
we get
\begin{align*}
T_1(x) & = \frac{2}{\pi-x} + \sum_{k\geq 1}\frac{4x^{2}}{(\pi-x)(x^2-(2k\pi)^2)}\,.
\end{align*}
Taking two derivatives we obtain, for $x \in (0,\pi)$,
\begin{align}
T_1''(x) & = \frac{4}{(\pi-x)^3} 
+ \sum_{k\geq 1}\bigg(-\frac{8}{(4k^2-1) (\pi-x)^3} 
- \frac{8k}{(x-2k\pi)^3(2k-1)}\notag\\
&\qquad \qquad \qquad \qquad \qquad \qquad \qquad \qquad \qquad\qquad- \frac{8k}{(x+2k\pi)^3(2k+1)}\bigg) \notag\\[1mm]
& = -\sum_{k\geq 1}\left(\frac{8k}{(x-2k\pi)^3(2k-1)} + \frac{8k}{(x+2k\pi)^3(2k+1)}\right) \notag\\[1mm]
& = \sum_{k \geq 1} \frac{32 k^2 (4 k^2 \pi^{3} + 12 k^2 \pi^{2} x + 3 \pi x^2 + x^3)}{(2k-1) (1 + 2 k) (2 k \pi - x)^3 (2 k \pi + x)^3} > 0,\label{T1pp}
\end{align}
where we have used that the first term of the infinite sum is telescopic and cancels out with the term which is independent of $k$.

A straightforward computation yields
\begin{align*}
T_2''(x) & = 2 \pi \frac{6 (\pi - x) (1 - \cos x) - (12 - (\pi -
           x)^2) \sin x}{(\pi-x)^5}\,,
\end{align*}
so obviously $T_2''(x) \leq 0$ if and only if
\begin{align*}
6 (\pi - x) (1 - \cos x) - (12 - (\pi - x)^2) \sin x & \leq 0\,.
\end{align*}
When written in terms of the variable $\frac x2$, this can be readily shown to be
equivalent to demanding that
\begin{equation}\label{eqtan2}
\tan\frac{x}{2}  \leq \frac{12-(\pi-x)^2}{6(\pi-x)} = \frac{2}{\pi-x} - \frac{\pi-x}{6}\,.
\end{equation}
The Taylor expansion of this function at~$\pi$ can be read
off~\cite[4.3.70]{Abramowitz}:
\begin{align*}
\tan\frac{x}{2} & = \frac{2}{\pi-x} - 2\sum_{j\geq 0} \frac{(-1)^{j+1}(x-\pi)^{2j+1}}{(2j+2)!}B_{2j+2} \\
& = \frac{2}{\pi-x} - \frac{\pi-x}{6} - 2\sum_{j\geq 1} \frac{(-1)^{j}(\pi-x)^{2j+1}}{(2j+2)!}B_{2j+2} \,.
\end{align*}
Here $B_n$ denotes the $n$th Bernoulli number. Since $(-1)^{j}B_{2j+2}
> 0$ for every $j \geq 1$, we infer~\eqref{eqtan2} and the lemma follows.
\end{proof}

The last step is to show that
$\widetilde{G}_s(\frac{s-1}{s+1}\pi,\bar\te_2) > 0$. To this end we will
define the auxiliary one-variable function
\begin{align*}
\widehat{G}_s(\bar\te_2) :=
  \widetilde{G}_s\Big(\frac{s-2}{s}\pi,\bar\te_2\Big)\, \frac{2s\pi}{\pi-\bar\te_2}\,.
\end{align*}
A short computation shows that
\[
\widehat{G}_s'(\bar\te_2)  = 4s\pi T_1'(\bar\te_2) + 
( s-2) \pi s \sec^2\frac{\pi}{s} -s^3  \tan\frac{\pi}{s}
\]
and
\[
\widehat{G}_s''(\bar\te_2)  = 4s\pi T_1''(\bar\te_2)\,,
\]
which is positive on $(0,\pi)$ by~\eqref{T1pp}. As moreover
\begin{align*}
\widehat{G}_s\Big(\frac{s-2}{s}\pi\Big)= \widehat{G}_s'\Big(\frac{s-2}{s}\pi\Big) = 0\,,
\end{align*}
we infer that $\widehat{G}_s(\bar\te_2) > 0$ for all $\bar\te_2
< \frac{s-2}{s}\pi$, as we wanted to
prove. Proposition~\ref{propderivadaGpositiva} then follows.
\end{proof}

Two immediate consequences of Proposition~\ref{propderivadaGpositiva}
are that there are no critical points of $G_s$ in the interior of
$\cS_s$ and that the minimum of $G_s$ in $\cR_s$ is attained on $\Gamma^{\cR_s}_{1} \cup \Gamma^{\cR_s}_{4}$.
Furthermore,
\begin{align}
\min_{(\bar\te_1,\bar\te_2) \in \cS_s} G_s(\bar\te_1,\bar\te_2) = \min_{(\bar\te_1,\bar\te_2) \in \Gamma_{1}^{\cS_s}} G_s(\bar\te_1,\bar\te_2) \geq \min_{(\bar\te_1,\bar\te_2) \in \Gamma_{1}^{\cR_s}} G_s(\bar\te_1,\bar\te_2)\,,\label{minima}
\end{align}
with equality if and only if the RHS minimum is attained at 
\[
(\bar\te_1,\bar\te_2) =
\Big(\frac{s-1}{s+1}\pi,0 \Big)\quad \text{or}\quad
(\bar\te_1,\bar\te_2) = \Big (\frac{s-1}{s+1}\pi,\frac{s-2}{s}\pi \Big)\,.
\]

In either of the two cases, the minimum over $\cS_{s}$ is attained at only one single point, respectively:
\[
(\bar\te_1,\bar\te_2) =
\Big(\frac{s-1}{s+1}\pi,0 \Big)\quad \text{or}\quad
(\bar\te_1,\bar\te_2) = \Big (\pi,\frac{s-2}{s}\pi \Big)\,.
\]

We shall next show that the minimum is indeed attained at the latter
point:

\begin{lemma}\label{lemma2}
Let 
\begin{align*}
g_s(\bar\te_2) & := G_s \Big(\frac{s-1}{s+1}\pi,\bar\te_2 \Big)
\end{align*}
be the restriction of~$G_s$ to the boundary $\Gamma_{1}^{\cR_s}$. The
minimum of $g_s$ in the interval $[0,\frac{s-2}s\pi]$ is attained at
the endpoint $\bar\te_2 = \frac{s-2}{s}\pi$.
\end{lemma}

\begin{proof}
We will show that the function $g_s$, which can be written as
\[
g_s(\bar\te_2)=\frac{4 \pi \bar\te_2 \cot\frac{\bar\te_2}{2} + (s^2 -1) (\pi-\bar\te_2) ((s-2 ) \pi - s \bar\te_2)\tan\frac{\pi}{s+1}}{(s - 2) (\pi - \bar\te_2) ((s-1 ) \pi - (1 + s) \bar\te_2)}\,,
\]
is a decreasing function of $\bar\te_2\in [0,\frac{s-2}2\pi]$. If we
take a derivative we find
\begin{align*}
g_s'(\bar\te_2) = - 2\pi \frac{\widehat g_s(\bar\te_2) + (s^2 - 1)\tan\frac{\pi}{1+s}}{(s-2 ) (\bar\te_2(1+s) +  \pi(1-s))^2}\,.
\end{align*}
where we have set
\begin{align*}
\widehat g_s(\bar\te_2) & := \frac{2 (\bar\te_2^2 (1 + s) - (s-1) \pi^{2}) \cot\frac{\bar\te_2}{2} +  \bar\te_2(\bar\te_2 - \pi) (\bar\te_2(1+s) + \pi(1-s)) \csc^2\frac{\bar\te_2}{2}}{(\bar\te_2-\pi)^{2}}\,.
\end{align*}

We will now compute the minimum of $\widehat g_s$ by taking its
derivative and locating its zeros:
\begin{align*}
\widehat g_s'(\bar\te_2) & = -\frac{((s-1 ) \pi - (1 + s) \bar\te_2) \csc^2\frac{\bar\te_2}{2} }{(\pi-\bar\te_2)^3} \Big[(\pi - \bar\te_2)^2 \bar\te_2\cot\frac{\bar\te_2}{2} + 2 \pi (\bar\te_2 + \sin\bar\te_2 -\pi )\Big]
\end{align*}
By Lemma \ref{lemmaT1T2}, the second factor is strictly positive
on~$(0,\pi]$ and thus
\begin{align*}
\widehat g_s'(\bar\te_2) < 0\,,
\end{align*}
which implies that
\begin{align*}
 \min_{\bar\te_2\in [0,\frac{s-2}s\pi]} \widehat g_s(\bar\te_2) + (s^2 - 1)\tan\frac{\pi}{1+s}  &= \widehat g_s\Big(\frac{s-2}{s}\pi\Big) + (s^2 - 1)\tan\frac{\pi}{1+s}\\
& > \widehat g_s\Big(\frac{s-1}{s+1}\pi\Big) + (s^2 - 1)\tan\frac{\pi}{1+s}= 0\,.
\end{align*}
This shows that $g_s(\bar\te_2)$ is decreasing, thereby finishing the proof.
\end{proof}

By~\eqref{minima} and Lemma~\ref{lemma2}, we then have that
\begin{align*}
\min_{(\bar\te_1,\bar\te_2) \in \cS_s} G_s(\bar\te_1,\bar\te_2) = G_s \Big (\pi,\frac{s-2}{s}\pi \Big)\,,
\end{align*}
and in fact $(\pi,\frac{s-2}{s}\pi) $ is the only
minimizer in~$\cS_s$. This implies that the only possible solution of
$P'_{s,n}$ has some of its angles equal to $\pi$ (since $k$ and $n-k$
are both greater than 0), and therefore we dealt with it before as
part of the boundary cases of $P_{s,n}$. This finishes the proof.
\end{proof}

\begin{remark}
 The convexity of the polygon is a necessary condition, since one can
 construct non-convex maximizers which are not semi-regular by
 appending a straight ``hair'' (that is, a long spike of
 negligible area) to a convex polygon. This allows to construct arbitrary values of $S(P)$ for any $|P| < \frac{1}{4\pi}$ by changing the length of the hair and its interior angle. See Figure \ref{F.Pelo}.
\end{remark}

\begin{figure}
\centering
\includegraphics[scale=0.4]{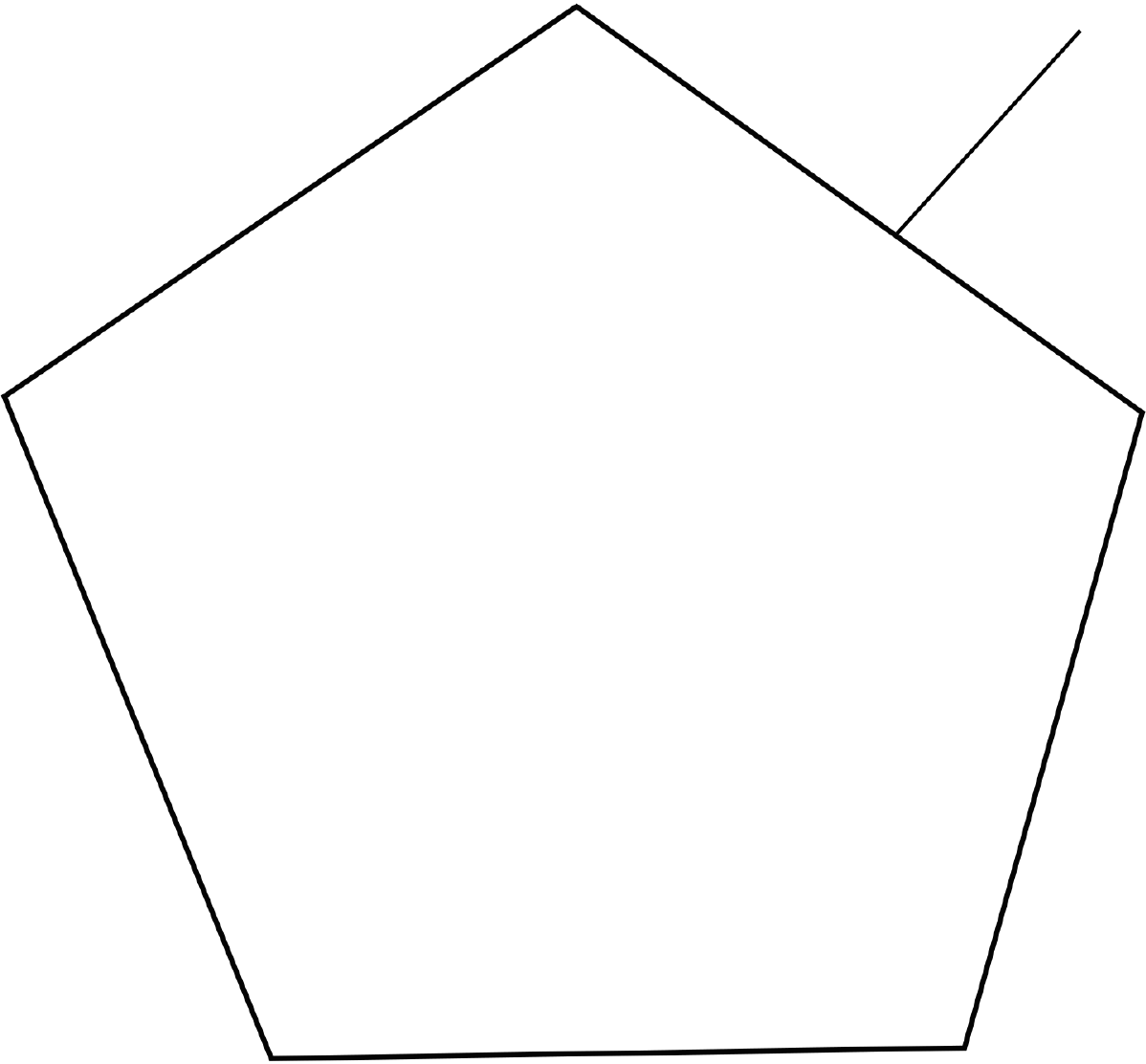}
\caption{Regular pentagon with a hair.}
\label{F.Pelo}
\end{figure}

An immediate consequence of Theorem~\ref{T.maximizers} is that
circumscriptible polygons as above (in particular, $n$-sided regular
polygons) are spectrally determined in the class of convex polygons:

\begin{corollary}\label{C.maximizers}
Let $P_n$ be a circumscriptible $n$-sided polygon whose angles are all
equal but possibly one, which must then be larger than the rest. If a
convex polygon~$P$ has the same Dirichlet or Neumann spectrum
as~$P_n$, then it is congruent to it.
\end{corollary}

%\begin{proof}
%
%\end{proof}

\section{Reduction to a problem for polygons and conclusion of the proof}
\label{S.strategy}

In view of Corollary~\ref{C.maximizers}, Theorem~\ref{T.main} will
follows once we prove that a convex piecewise smooth domain (possible
with straight corners) that is isospectral
to a polygon must be a polygon too. This is an easy consequence of the
fact that the fourth order term in the asymptotic expansion of the
(Dirichlet or Neumann) heat trace is zero if and only if the boundary
of the domain consists of segments:

\begin{proposition}\label{P.heat}
Let $\Om$ be a piecewise smooth domain with $m\geq0$ straight
corners. Then the asymptotic behavior of its Dirichlet heat trace at small times
is
\[
h_\Om\Dir(t)=\frac{|\Om|}{4\pi t} -\frac{|\pd\Om|}{8\sqrt{\pi
    t}}
+\frac1{12\pi}\bigg(\int_{\pd\Om}\ka\,ds+\sum_{k=1}^m\frac{\pi^2-\te_k^2}{2\te_k}\bigg)%\\
+\frac{\sqrt t}{256\sqrt\pi} \int_{\pd\Om}\ka^2\, ds + O(t)\,,
\]
where $\te_k\in(0,2\pi)$ is the interior angle at the $j\th$ corner
point, $\ka$ is the curvature of the boundary at each
differentiable point and~$s$ is an arc-length parameter. Likewise, in the case
of Neumann boundary conditions one has
\[
h_\Om\Neu(t)=\frac{|\Om|}{4\pi t} +\frac{|\pd\Om|}{8\sqrt{\pi
    t}}
+\frac1{12\pi}\bigg(\int_{\pd\Om}\ka\,ds+\sum_{k=1}^m\frac{\pi^2-\te_k^2}{2\te_k}\bigg)%\\
+\frac{5\sqrt t}{256\sqrt\pi} \int_{\pd\Om}\ka^2\, ds + O(t)\,.
\]
\end{proposition}
\begin{proof}
If the domain is smooth, the coefficients of the asymptotic expansion
of the heat trace can be found in~\cite{Gilkey}. In the
Dirichlet case, the
contribution of a straight corner was first reported by Ray
and a full proof can be found in~\cite{Berg}. In the Neumann case, the
contribution of a straight corner can be found in~\cite{Mazzeo}. As is
customary, the
sum of the coefficients of the smooth domain and of the straight
corner gives the final formula presented in the statement.
\end{proof}

An immediate corollary of Proposition~\ref{P.heat}, which together
with Corollary~\ref{C.maximizers} completes the proof of
Theorem~\ref{T.main}, is the following:

\begin{corollary}\label{C.heat}
A piecewise smooth domain with straight corners  to a polygon must be a polygon.
\end{corollary}
\begin{proof}
Since the term of order $\sqrt t$ in the asymptotic expansion of the
heat trace of~$\Om$ must vanish, we infer that
\[
\int_{\pd\Om}\ka^2\, ds=0\,,
\]
i.e., that the boundary of~$\Om$ is flat.  Hence it must consist of
segments, so $\Om$ is a polygon.
\end{proof}

\section*{Acknowledgments}

A.E.\ is supported by the ERC Starting Grant~633152. J.G.-S. was
partially supported by an AMS-Simons Travel Grant, by the grant
MTM2014-59488-P (Spain) and by the Simons Collaboration
Grant~524109. Both authors were partially supported by the
ICMAT--Severo Ochoa grant SEV-2015-0554. We would like to thank Michel van den Berg, Jeff
Cheeger, Peter Gilkey, Fabricio Maci\`a, Daniel Peralta-Salas, Julie Rowlett and David Sher for helpful discussions.

\bibliographystyle{amsplain}

\end{document}